\documentclass[12pt]{amsart}
\usepackage{amssymb}
\usepackage[all]{xy}
\newtheorem{theorem}{Theorem}[section]
\newtheorem{lemma}[theorem]{Lemma}

\newtheorem{remark}[theorem]{Remark}
\newtheorem{corollary}[theorem]{Corollary}

\def\to{\rightarrow}
\def\f{\mathfrak}
\def\m{\mathbb}

\def\b{\mathbf}
\def\r{\mathrm}
\def\ot{\otimes}
\begin{document}
\baselineskip15pt
\title[Noncommutative mapping schemes]{On the structure of noncommutative\\ mapping schemes}
\author[M.M. Sadr]{Maysam Maysami Sadr}
\address{Department of Mathematics\\
Institute for Advanced Studies in Basic Sciences\\
P.O. Box 45195-1159, Zanjan 45137-66731, Iran}
\email{sadr@iasbs.ac.ir}
\begin{abstract}
The following three types of objects are considered in a dual functorial formalism:
(i) ind-scheme of mappings between two schemes, (ii) for a quantum group $G$, ind-scheme of $G$-mappings between two $G$-schemes, and
(iii) ind-scheme of group homomorphisms between two quantum group. By schemes and quantum groups here we mean objects which are
respectively dual to unital associative algebras and Hopf algebras.
\end{abstract}
\maketitle
\section{Introduction}
The main goal of this note is to describe structure of the following types of objects in a \emph{dual} and \emph{constructive} functorial formalism:
\begin{enumerate}
\item[(i)] ind-scheme of mappings between two schemes;
\item[(ii)] for any quantum group $G$, ind-scheme of $G$-mappings between two $G$-schemes, and
\item[(iii)] ind-scheme of group homomorphisms between two quantum group.
\end{enumerate}
Here, by \emph{scheme}, we mean an object which is dual to a unital associative algebra over a field.
In other words, any algebra $A$ is considered as the algebra of \emph{polynomial functions} on a scheme $\f{S}A$
($\f{S}$ stands for scheme, space or spectrum).
In Noncommutative Algebraic Geometry such an object $\f{S}A$ is called \emph{affine noncommutative scheme}; see, for instance,
\cite{KontsevichRosenberg1}. Similarly, by \emph{quantum group} or \emph{quantum group scheme},
we mean an object which is dual to a Hopf-algebra. In \cite{Sadr1}, we have considered these three types of
\emph{noncommutative mapping schemes} in the case that
domains of the mappings are \emph{finite} schemes, i.e. schemes dual to finite dimensional algebras. The idea behind of our constructions is
very well-known and simple: For spaces $Y,Z$ that have a specific structure, a map $f:Z\times X\to Y$ satisfying some appropriate properties,
can be considered as a family of the structure preserving mappings from $Z$ to $Y$, parameterized by $X$. Then, the space of all such structure
preserving mappings, if exists, must be a universal parameterizing space. Following this idea, For two algebras $B,C$, we can define the
\emph{algebra of polynomial functions on the scheme of all mappings from $\f{S}C$ to $\f{S}B$}, as an algebra $A$ together with an algebra
morphism $h:B\to C\ot A$ satisfying the following universal property: For every algebra $E$ and algebra morphism $e:B\to C\ot E$, there is
a unique algebra morphism $\hat{e}:A\to E$ satisfying $e=(\r{id}_C\ot\hat{e})h$. For more details see \cite{Sadr1,Sadr2,Sadr3,Soltan1}.
In \cite{Sadr1}, we showed that the universal algebra $A$ exists provided that $C$ is finite dimensional. In Section \ref{5-141626} of this note,
we show, by a constructive method, that if the algebra $A$ is replaced by a pro-algebra then the universal family $h:B\to C\ot A$
exists provided that $B$ is finitely generated. Also, we show that the assignment $(B,C)\mapsto A$ is functorial with respect to both of the
$B$ and $C$, see Theorem \ref{5-131325} and Corollary \ref{1711281216}.
This type of pro-algebras and the associated ind-schemes has been considered before by some authors for commutative algebras
and ordinary affine schemes, for instance see \cite{Shafarevich1,Shafarevich2,CherenaeckGuerra1,Kambayashi1,KambayashiMiyanishi1}.
(See also \cite[Section 3]{Gersten1} the noncommutative case.)
Our approach to this type of objects is very more general and seemed to be very more simple. (Also, our constructions are easily \emph{computable},
see the example at the end of Section \ref{1712071723}.) In Section \ref{1712071723} we generalize a result of
\cite{Sadr1} by showing that the ind-scheme of all mappings from a scheme to a quantum group has a canonical quantum group structure.
In Sections \ref{1712031224} and \ref{1712071728} we consider the construction of the objects of types (ii) and (iii) mentioned above.
Although some results stated in Section \ref{5-141626} have been known (at least in the commutative case) since many years ego,
but we emphasize that all the other results are new.

\textbf{Notations \& Terminology.}
Throughout, categories are denoted by bold letters.
Let $\b{C}$ be a category. For objects $C,C'$ in $\b{C}$, $\b{C}(C,C')$ denotes the set of morphisms in $\b{C}$ from $C$ to $C'$
and $\r{id}_C\in\b{C}$ denotes the identity morphism on $C$.
The categories of \emph{pro-objects} and \emph{ind-objects} of $\b{C}$ are denoted respectively by $\b{proC}$ and $\b{indC}$.
(For general theory of pro and ind objects see \cite[Appendix]{ArtinMazur1}.)
An object in $\b{proC}$ is a contravariant functor from a directed set (considered as a category in the usual way)
to $\b{C}$. Thus any object $C\in\b{proC}$ is distinguished by an indexed
family $\{C_i\}_{i\in\b{I}}$ of objects in $\b{C}$, where $(\b{I},\leq)$ is a directed set, together with a morphism
$C^{i'}_i\in\b{C}(C_{i'},C_i)$ for every $i,i'\in\b{I}$ with $i\leq i'$, such that $C^i_i=\r{id}_{C_i}$ and
$C^{i''}_i=(C^{i'}_i)(C^{i''}_{i'})$ for $i\leq i'\leq i''$. If $C,D$ are pro-objects of $\b{C}$ indexed respectively by $\b{I},\b{J}$
then the pro-morphism set is defined by
$$\b{proC}(C,D):=\underleftarrow{\lim}_j\underrightarrow{\lim}_i\b{C}(C_i,D_j).$$
The above definition of pro-morphisms can be expressed in a simple way as follows. A \emph{represented pro-morphism}
from $C$ to $D$, is a subset $\Phi$ of $\cup_{i,j}\b{C}(C_i,D_j)$ satisfying the following two conditions:
\begin{enumerate}
\item[i)] For every $j$ there exist $j'\geq j$ and $i$ such that $\Phi\cap\b{C}(C_i,D_{j'})\neq\emptyset$.
\item[ii)] $\Phi$ is compatible; this means that if $f:C_i\to D_j$ and $f':C_{i'}\to D_{j'}$ are in $\Phi$
and if $j\leq j'$ then there exists $i''\geq i,i'$ such that $fC^{i''}_i=D^{j'}_{j}f'C^{i''}_{i'}$.
\end{enumerate}
Two represented pro-morphisms $\Phi,\Psi$ from $C$ to $D$ are called \emph{equivalent} if $\Phi\cup\Psi$ is compatible.
The equivalence classes of represented pro-morphisms from $C$ to $D$ are in one-to-one correspondence (and hence identified) with
the elements of $\b{proC}(C,D)$. We denote by $[\Phi]$ the pro-morphism containing $\Phi$.
Note that if $f:C_i\to D_j$ is in $\Phi$ and $j'\leq j$ then $\Phi\cup\{D^j_{j'}f\}\in[\Phi]$. This enables us to replace $\Phi$
with an equivalent represented pro-morphism $\Phi'$ with the property that for every $j$ there is $i$ with $\Phi'\cap\b{C}(C_i,D_j)\neq\emptyset$.
And so we can define the composition of pro-morphisms in the obvious manner. Any object in $\b{C}$ can be considered as a pro-object over
a directed set with only one element. Thus we identify $\b{C}$ as a full subcategory of $\b{proC}$.

Throughout, we work over a fixed field $\m{K}$. Algebras, vector spaces, linear maps, and tensor products are all over $\m{K}$.
Algebras have units and algebra morphisms preserve units. The category of algebras is denoted by $\b{A}$.
The full subcategory of commutative algebras is denoted by $\b{A}_c$.
The opposite category $\b{S}:=\b{A}^\r{op}$ is called category of \emph{schemes}. Any morphism in $\b{S}$ is called a \emph{mapping}.
We let $\b{S}_\r{c}:=\b{A}_\r{c}^\r{op}$. The subscripts `$\r{fd}$' and `$\r{fg}$' for a category of algebras, show that objects of the category
are respectively finite dimensional and finitely generated algebras. The subscripts `$\r{fnt}$' and `$\r{fd}$' for a category of schemes,
show that objects of the category are respectively finite and finite dimensional schemes; indeed, it is assumed that $\b{S}_\r{fnt}:=\b{A}^\r{op}_\r{fd}$
and $\b{S}_\r{fd}:=\b{A}^\r{op}_\r{fg}$. It follows from the above definitions that $\b{indS}=\b{proA}^\r{op}$.
The category of Hopf algebras \cite{Sweedler1} is denoted by $\b{H}$.
The objects of $\b{G}:=\b{H}^\r{op}$ are called \emph{quantum group schemes}. Any morphism in $\b{G}$ is called \emph{group homomorphism}.
Note that any commutative Hopf algebra is dual to an ordinary group schemes \cite{Milne1}.

Throughout, the symbol $\f{S}$ denotes the renaming duality functor from any category of algebras to the corresponding category of schemes.
For an algebra $A$, by a \emph{point} of scheme $\f{S}A$ we mean an algebra morphism from $A$ to $\m{K}$. Similarly, for a pro-algebra $A$,
a point of the ind-scheme $\f{S}A$ is a pro-morphism form $A$ to $\m{K}$.

Let $A,B$ be pro-algebras in $\b{proA}$ indexed respectively by $\b{I},\b{J}$. The tensor product $A\ot B$ is a pro-algebra indexed by
$\b{I}\times\b{J}$ and defined by $(A\ot B)_{(i,j)}:=A_i\ot B_j$ and $(A\ot B)^{(i',j')}_{(i,j)}:=A^{i'}_i\ot B^{j'}_j$.
\section{Algebras of polynomial functions on mapping schemes}\label{5-141626}
First of all we define three categories $\b{B},\b{C},\b{D}$ associated to $\b{A}$, which simplify statements of our results.
The objects in $\b{B}$ are pairs $(B:G)$, denoted shortly by $B_G$,
where $B$ is an object in $\b{A}$ and $G\subseteq B$ generates $B$ as an algebra.
A morphism in $\b{B}$ from $B_G$ to $B'_{G'}$ is a morphism $g:B\to B'$ in $\b{A}$ such that $g(G)\subseteq G'$.
The objects in $\b{C}$ are pairs $(C:L)$, denoted shortly by $C_L$, where $C$ is an object in $\b{A}$ and
$L$ is a finite linearly independent subset of $C$.
A morphism in $\b{C}$ from $C_L$ to $C'_{L'}$ is a morphism $f:C\to C'$ in $\b{A}$ such that $f(L)$ is a subset of the linear span of $L'$
in $C'$. The objects in $\b{D}$ are triples $D=(B,C,\{\delta_b\}_{b\in G})$
where $B,C$ are objects in $\b{A}$, $G\subseteq{B}$ generates ${B}$ as an algebra, and $\delta_b$ is a
linearly independent finite subset of $C$ for every $b\in G$. A morphism $D\to D'=({B}',C',\{\delta'_{b'}\}_{b'\in G'})$ in $\b{D}$
is a pair $(g,f)$ where $g\in\b{A}({B},{B}'),f\in\b{A}(C',C)$ such that $g(G)\subseteq G'$ and for every $b\in G$,
$f(\delta'_{g(b)})$ is a subset of the linear span of $\delta_b$. Composition of morphisms in $\b{B},\b{C},\b{D}$ is defined in the obvious way.
For objects $B_G,C_L$ respectively in $\b{B},\b{C}$ we let $\f{d}(B_G,C_L)$ denote the object in $\b{D}$
defined by the triple $(B,C,\{\delta_b\}_{b\in G})$ where $\delta_b:=L$ for every $b\in G$. If $g:B_G\to B'_{G'},f:C'_{L'}\to C_L$
are morphisms respectively in $\b{B},\b{C}$ then we let $\f{d}(g,f)$ be the morphism $(g,f)$ in $\b{D}$ from
${\f{d}}(B_G,C_L)$ to ${\f{d}}(B'_{G'},C'_{L'})$. Then ${\f{d}}$ can be considered as a functor:
\begin{equation}\label{5-291249}
\f{d}:\b{B}\times\b{C}^\r{op}\to\b{D}.
\end{equation}
\begin{lemma}\label{5-121553}
Let $D=({B},C,\{\delta_b\}_{b\in G})$ be an object in $\b{D}$. Then there exist an algebra $\tilde{\f{A}}(D)=A$
and a morphism $\tilde{\f{h}}(D)=h:{B}\to C\ot A$ such that for every $b\in G$, $h(b)$ is a sum of elements of the form $c\ot a$
with $c\in\delta_b,a\in A$, and such that the pair $(A,h)$ is universal with respect to this property, that is, if $A'$ is an algebra
and $h':{B}\to C\ot A'$ is an algebra morphism such that for every $b\in G$, $h'(b)$
is a sum of elements of the form $c\ot a'$ with $c\in\delta_b$, then there is a unique morphism
$\alpha:A\to A'$ in $\b{A}$ such that $h'=(\r{id}_C\ot \alpha)h$.
\end{lemma}
\begin{proof}
The desired algebra $A$ is the universal algebra in $\b{A}$ generated by the set of symbols
$\{x_{b,c}:b\in G,c\in \delta_b\}$ such that the assignment $b\mapsto\sum_{c\in\delta_b}c\ot x_{b,c}$
defines an algebra morphism $h$ from ${B}$ to $C\ot A$.
\end{proof}
Let $(g,f):D\to D'$ be a morphism in $\b{D}$. Then there is a unique morphism $\alpha:\tilde{\f{A}}(D)\to\tilde{\f{A}}(D')$
in $\b{A}$ satisfying $(\r{id}\ot\alpha)\tilde{\f{h}}(D)=(f\ot\r{id})\tilde{\f{h}}(D')g$. We denote $\alpha$ by $\tilde{\f{A}}(g,f)$.
\begin{lemma}\label{p2}
The assignments $D\mapsto\tilde{\f{A}}(D),(g,f)\mapsto\tilde{\f{A}}(g,f)$ define a functor
$\tilde{\f{A}}:\b{D}\to\b{A}$,
and the assignment $D=(B,C,\{\delta_b\}_{b\in G})\mapsto[\tilde{\f{h}}(D):B\to C\ot\tilde{\f{A}}(D)]$ defines a natural transformation.
\end{lemma}
\begin{proof}
It follows from the universal property of pairs $(\tilde{\f{A}}(D),\tilde{\f{h}}(D))$.
\end{proof}
We have the following corollary of Lemmas \ref{p2} and \ref{5-121553}.
\begin{theorem}\label{5-301727}
There exist a functor $\f{A}$ and a natural transformation $\f{h}$,
\begin{equation}\label{5-211732}
\f{A}:\b{B}\times\b{C}^\r{op}\to\b{A},\quad(B_G,C_L)\mapsto[\f{h}(B_G,C_L):B\to C\ot\f{A}(B_G,C_L)],
\end{equation}
such that for $b\in G$, $\f{h}(B_G,C_L)(b)$ is a linear combination of elements of the form $c\ot a$ with $c\in L$, and the pair
$(\f{A}(B_G,C_L),\f{h}(B_G,C_L))$ is universal with respect to this latter property.
\end{theorem}
The universal property offers that $\f{S}\f{A}(B_G,C_L)$ to be called
\emph{scheme of mappings of type $(L,G)$ from $\f{S}C$ to $\f{S}B$}.
\begin{proof}
Let $\f{A}$ be the composition $\tilde{\f{A}}\f{d}$ and let $\f{h}(B_G,C_L):=\tilde{\f{h}}\f{d}(B_G,C_L)$ where $\f{d}$ is
the functor described in (\ref{5-291249}). Then the theorem follows easily from Lemmas \ref{p2} and \ref{5-121553}.
\end{proof}
Let $B\in\b{A}$ and $C\in\b{A}_\r{fd}$. Suppose that $G,G'\subset B$ be two generating set for $B$ and $V,V'$ be two vector basis for $C$.
It follows from the universal property described in Theorem \ref{5-301727} that the algebras $\f{A}(B_G,C_V)$ and $\f{A}(B_{G'},C_{V'})$
are canonically isomorphic. Thus, $\f{A}$ can be considered as a functor from $\b{A}\times\b{A}^\r{op}_\r{fd}$ to $\b{A}$:
\begin{theorem}\label{1711241959}
There exist a functor $\f{A}$ and a natural transformation $\f{h}$,
\begin{equation}\label{5-131237}
\f{A}:\b{A}\times\b{A}^\r{op}_\r{fd}\to\b{A},\quad\quad(B,C)\mapsto[\f{h}(B,C):B\to C\ot\f{A}(B,C)],
\end{equation}
such that for every algebra $E$ and morphism $e:B\to C\ot E$ there is a unique morphism $\hat{e}:\f{A}(B,C)\to E$
satisfying $e=(\r{id}\ot\hat{e})\f{h}(B,C)$.
\end{theorem}
Note also that there is a canonical one-to-one correspondence between $\b{A}(\f{A}(B,C),\m{K})$ and $\b{A}(B,C)$ induced by the universal
property of Theorem \ref{5-301727}. For more details see \cite{Sadr1}.
\begin{proof}
It follows immediately from Theorem \ref{5-301727}.
\end{proof}
Let $\f{M}$ denote the functor induced by $\f{A}$ on the opposite categories. We have the following immediate corollary of Theorem \ref{1711241959}.
\begin{corollary}\label{1711281203}
We have a canonical functor $\f{M}:\b{S}_\r{fnt}^\r{op}\times\b{S}\to\b{S}$ such that for $S_1\in\b{S}_\r{fnt},S_2\in\b{S}$ there is a natural
bijection between the set of mappings from $S_1$ to $S_2$ and the set of points of $\f{M}(S_1,S_2)$. Thus, $\f{M}(S_1,S_2)$ is called the \emph{scheme
of mappings from $S_1$ to $S_2$}.
\end{corollary}
Now we show that if the values of $\f{A}$ in (\ref{5-211732}) and (\ref{5-131237})
are allowed to belong to $\b{proA}$ then $\f{A}$ can be extended to
$\b{A}_\r{fg}\times \b{A}^\r{op}$. Let $B\in\b{A}_\r{fg}$ and $C\in\b{A}$. Suppose that $G\subset B$ is a finite generating set for $B$,
and $V$ is a vector basis for $C$. Let $\b{J}$ denote the directed set of all finite subsets of $V$ with inclusion as the ordering.
We define a pro-algebra ${\f{A}}(B,C)=A$ indexed by $\b{J}$, and a pro-morphism ${\f{h}}(B,C)=h:B\to C\ot A$ in $\b{proA}$, as follows.
For every $L\in\b{J}$ let $A_L:={\f{A}}(B_G,C_L)$. For $L_1\subseteq L_2\in\b{J}$, $(\r{id},\r{id})$ defines a morphism in
$\b{B}\times\b{C}^\r{op}$ from $(B_G,C_{L_2})$ to $(B_G,C_{L_1})$. Let $A^{L_2}_{L_1}:={\f{A}}(\r{id},\r{id})$. Consider the set
of morphisms $\Phi:=\{\phi_L:L\in\b{J}\}$ where $\phi_L:={\f{h}}(B_G,C_L)$. Then $\Phi$ is a represented pro-morphism
from $B$ to $C\ot A$. Set $h:=[\Phi]$. We now show that the pair $(A,h)$ has a universal property analogous to the one mentioned in
Theorem \ref{1711241959}: Let $E$ be a pro-algebra indexed by $\b{I}$ and $e:B\to C\ot E$ be a pro-morphism in $\b{proA}$. Let $\Psi$ be a represented
pro-morphism such that $e=[\Psi]$. Let $\psi:B\to C\ot E_i$ belongs to $\Psi$. Since $G$ is finite, there is $L_\psi\in J$
such that for every $b\in G$, $\psi(b)$ is a linear combination of elements of the form $c\ot x$ with $c\in L_\psi$. By the universal property
of $(A_{L_\psi},\phi_{L_\psi})$ there is a unique morphism $\omega_\psi:A_{L_\psi}\to E_i$ such that $\psi=(\r{id}\ot\omega_\psi)\phi_{L_\psi}$.
Let $\Omega:=\{\omega_\psi:\psi\in\Psi\}$. Then $\Omega$ is a represented pro-morphism from $A$ to $E$ and we have $e=(\r{id}\ot[\Omega])h$.
Also it is not hard to see that if a represented pro-morphism $\Omega'$ from $A$ to $E$ satisfies $e=(\r{id}\ot[\Omega'])h$ then
$\Omega\cup\Omega'$ is compatible and thus $[\Omega']=[\Omega]$. This finishes the establishment of the universal property of $(A,h)$.
This property has three important consequences. The first one is that the isomorphism class of ${\f{A}}(B,C)$ in $\b{proA}$
does not depend on the specific choice of $G$ or $V$. The second one is that if $g:B\to B',f:C'\to C$ are morphisms in $\b{A}$ where $B'$
is finitely generated then there is a unique pro-morphism ${\f{A}}(g,f):{\f{A}}(B,C)\to{\f{A}}(B',C')$ such that
$$(f\ot\r{id}){\f{h}}(B',C')g=(\r{id}\ot{\f{A}}(g,f)){\f{h}}(B,C).$$
Dissembling some Set theoretical difficulties which are removable by working in a fixed universe, the third consequence
is that the assignments $(B,C)\mapsto{\f{A}}(B,C),(g,f)\mapsto{\f{A}}(g,f)$ define a covariant functor from the bicategory
$\b{A}_\r{fg}\times\b{A}^\r{op}$ to $\b{proA}$. Moreover, the construction shows that $\f{A}$ takes actually values in $\b{proA}_\r{fg}$.
So, we have proved:
\begin{theorem}\label{5-131325}
There exist a functor $\f{A}$ and a natural transformation $\f{h}$,
\begin{equation}\label{5-211757}
\f{A}:\b{A}_\r{fg}\times\b{A}^\r{op}\to\b{proA}_\r{fg},\quad(B,C)\mapsto[\f{h}(B,C):B\to C\ot\f{A}(B,C)],
\end{equation}
with the following universal property: If $E\in\b{proA}$ and $e:B\to C\ot E$ is a pro-morphism
then there is a unique pro-morphism $\hat{e}:\f{A}(B,C)\to E$ satisfying $e=(\r{id}\ot\hat{e})\f{h}(B,C)$.
\end{theorem}
We have the following immediate corollary of Theorem \ref{5-131325}.
\begin{corollary}\label{1711281216}
We have a canonical functor $\f{M}:\b{S}^\r{op}\times\b{S}_\r{fd}\to\b{indS}_\r{fd}$ such that for $S_1\in\b{S},S_2\in\b{S}_\r{fd}$ there is a natural
bijection between the set of mappings from $S_1$ to $S_2$ and the set of points of $\f{M}(S_1,S_2)$.
Thus, $\f{M}(S_1,S_2)$ is called the \emph{ind-scheme of mappings from $S_1$ to $S_2$}.
\end{corollary}
\begin{remark}
\emph{\begin{enumerate}
\item[(i)] Let $\b{A}_\r{cfgred}$ denote the category of commutative finitely generated reduced algebras.
In the case that $\m{K}$ is algebraically closed, $\b{V}:=\b{A}_\r{cfgred}^\r{op}$ is equivalent to the category of ordinary affine varieties over $\m{K}$.
Indeed, for any algebra $A$ in $\b{A}_\r{cfgred}$, there is a canonical bijection between $\b{A}(A,\m{K})$ and closed
points of the prim spectrum of $A$.
\item[(ii)] Let $\f{c}:\b{A}\to\b{A}_\r{c}$ be the functor that associates to any algebra its quotient by the commutator ideal,
and let $\f{r}:\b{A}_\r{c}\to\b{A}_\r{cred}$ be the functor that associates to any commutative algebra its quotient by the nil radical.
We denote by the same symbols the canonical extensions $\f{c}:\b{proA}_\r{fg}\to\b{proA}_\r{cfg}$ and $\f{r}:\b{proA}_\r{c}\to\b{proA}_\r{cred}$.
Then, it is clear that the functors
$$\f{cA}:\b{A}_\r{fg}\times\b{A}^\r{op}\to\b{proA}_\r{cfg},\quad\f{rcA}:\b{A}_\r{fg}\times\b{A}^\r{op}\to\b{proA}_\r{cfgred},$$
have universal properties as in Theorem \ref{5-131325}. We denote the dual functors by,
$$\f{M}_\r{c}:\b{S}^\r{op}\times\b{S}_\r{fd}\to\b{indS}_\r{cfd},\quad\f{M}_\r{cr}:\b{S}^\r{op}\times\b{S}_\r{fd}\to\b{indV}.$$
In the case that $S_1,S_2\in\b{V}$, the structure of the \emph{ind-affine variety} $\f{M}_\r{cr}(S_1,S_2)$ has been considered by some authors.
See, for instance, \cite{CherenaeckGuerra1}, \cite[Theorem 2.3.3]{KambayashiMiyanishi1}. It seems that our approach is
very more simple than the others.
\end{enumerate}}
\end{remark}
At the end of this section we mention two basic properties of $\f{M}$. For ind-schemes $S_1=\f{S}B_1,S_2=\f{S}B_2$, we let
$S_1\times S_2:=\f{S}(B_1\ot B_2)$. Thus in $\b{indS}_\r{c}$, $\times$ coincides with the categorical notion of product.
\begin{theorem}
\begin{enumerate}
\item[(i)]For $S\in\b{S},S_1,S_2\in\b{S}_\r{fd}$ there is a canonical isomorphism in $\b{indS}_\r{c}$,
$$\f{M}_\r{c}(S,S_1\times S_2)\cong\f{M}_\r{c}(S,S_1)\times\f{M}_\r{c}(S,S_2).$$
\item[(ii)] (Exponential Law) For $S\in\b{S}_\r{fd},S_1\in\b{S},S_2\in\b{S}_\r{fnt}$, there is a canonical isomorphism,
$$\f{M}(S_1\times S_2,S)\cong\f{M}(S_1,\f{M}(S_2,S)).$$
\end{enumerate}
\end{theorem}
\begin{proof}
For (i) see \cite[Theorem 2.8]{Sadr1}, and for (ii) see \cite[Theorem 2.10]{Sadr1}.
\end{proof}
\section{Quantum group ind-schemes of mappings}\label{1712071723}
In \cite{Sadr1}, we proved that if $H$ is a Hopf algebra and $C$ is a finite dimensional commutative algebra then $\f{A}(H,C)$
has a canonical Hopf algebra structure; that is, in the dual notations, if $G$ is a quantum group scheme and $S\in\b{S}_\r{cfnt}$
is a finite scheme then $\f{M}(S,G)$ has a canonical quantum group scheme structure. We call $\f{M}(S,G)$ quantum group scheme
of mappings from $S$ to $G$. It is easily seen that this construction is also functorial:
\begin{theorem}\label{1711281603}
The functor in (\ref{5-131237}) can be considered as a functor $\f{A}:\b{H}\times\b{A}^\r{op}_\r{cfd}\to\b{H}$.
In dual notations, we have a canonical functor $\f{M}:\b{S}^\r{op}_\r{cfnt}\times\b{G}\to\b{G}$.
\end{theorem}
\begin{proof}
See Theorem 5.5 of \cite{Sadr1}.
\end{proof}
Now, in order to extend the content of Theorem \ref{1711281603}, we introduce the notion of \emph{Hopf pro-algebra}.
A Hopf pro-algebra is a pro-algebra $A$ with two pro-morphisms
$\Delta:A\to A\ot A,\epsilon:A\to\m{K}$, and a pro-linear map $T:A\to A$, such that $\Delta,\epsilon,T$ satisfy
the similar identities for, respectively, comultiplication, counit and antipode, in the category $\b{proA}$.
A morphism between two Hopf pro-algebras is a pro-morphism in $\b{proA}$ which respects comultiplications and counits (and hence antipodes).
We denote by $\b{PH}$ the category of Hopf pro-algebras. (Note that $\b{PH}$ is different from the category $\b{proH}$.)
The objects of $\b{IG}:=\b{PH}^\r{op}$ are called \emph{quantum group ind-schemes}. The proof of the following result is a straightforward
combination of proofs of \cite[Theorem 5.5]{Sadr1} and Theorem \ref{5-131325}, and hence is omitted.
\begin{theorem}\label{1711301556}
The functor in (\ref{5-211757}) can be considered as a functor $\f{A}:\b{H}_\r{fg}\times\b{A}^\r{op}_\r{c}\to\b{PH}$.
In dual notations, we have a canonical functor $\f{M}:\b{S}^\r{op}_\r{c}\times\b{G}_\r{fd}\to\b{IG}$. Moreover, for every $S\in\b{S}_\r{c}$ and
$G\in\b{G}_\r{fd}$ there is a natural bijection between the set of mappings from $S$ to $G$, and the points of $\f{M}(S,G)$.
\end{theorem}
As an example, we describe the quantum group ind-scheme $\f{M}(S,G)$ in a very simple case that $S$ is the ordinary affine $m$-space $\mathbb{A}_m$,
and $G$ is a \emph{pseudogroup} or \emph{matrix quantum group} in the sense of Woronowicz \cite{Woronowicz1}. We first restate the definition of
pseudogroup \cite{Woronowicz1} in a purely algebraic setting. A pseudogroup $G=\f{S}B$ is described by a pair $(B,(u_{kl}))$ where $B\in\b{A}_\r{fg}$
and $(u_{kl})$ is a $n\times n$ matrix with entries in $B$ such that,
\begin{enumerate}
\item[(i)] $B$ is generated by the entries of $(u_{kl})$,
\item[(ii)] the assignment $u_{kl}\mapsto\sum_{r=1}^{n}u_{kr}\ot u_{rl}$ defines an algebra morphism $\Delta:B\to B\ot B$, and
\item[(iii)] there is a linear anti-multiplicative map $T:B\to B$ satisfying $T^2=\r{id}_B$ and
$\sum_{r=1}^nT(u_{kr})u_{rl}=\sum_{r=1}^nu_{kr}T(u_{rl})=\delta_{kl}$, where $\delta_{kl}$ is Kronecker delta.
\end{enumerate}
It can be proved that the assignment $u_{kl}\mapsto\delta_{kl}$ defines an algebra morphism $\epsilon:B\to\m{K}$, and
$B$ is a Hopf algebra with comultiplication $\Delta$, counit $\epsilon$, and antipode $T$.

Let $G=\f{S}B$ be a pseudogroup as above, and let $S=\m{A}_m:=\f{S}C$ where $C=\m{K}[x_1,\ldots,x_m]$ is the commutative polynomial algebra
with $m$ variables. We are going to construct a model for Hopf pro-algebra $A=\f{A}(B,C)$. For every integer $p\geq1$, let $A_p$ denote the algebra
generated by the symbols $a_{p,k,l,i_1,\ldots,i_m}$, with $1\leq k,l\leq n,0\leq i_j,\sum_{j=1}^mi_j\leq p$, such that,
$$u_{kl}\mapsto\sum_{i_1+\cdots+i_m\leq p}x_1^{i_1}\cdots x_m^{i_m}\ot a_{p,k,l,i_1,\ldots,i_m},$$
defines an algebra morphism $h_p:B\to C\ot A_p$. The universality of $A_p$ shows that for $p\leq p'$, the assignment
$a_{p',k,l,i_1,\ldots,i_m}\mapsto z$, where $z=a_{p,k,l,i_1,\ldots,i_m}$ when $\sum_{j=1}^mi_j\leq p$, and $z=0$ when $\sum_{j=1}^mi_j>p$,
defines an algebra morphism $A^{p'}_p:A_{p'}\to A_p$. Thus, the family $\{A_p,A^{p'}_p\}$ of algebras and morphisms
defines the underlying pro-algebra of the Hopf pro-algebra $A$. For every $p$, the assignment,
$$a_{2p,k,l,i_1,\ldots,i_m}\mapsto\sum_{r=1}^n\sum_{t_j+s_j=i_j}a_{p,k,r,t_1,\ldots,t_m}\ot a_{p,r,l,s_1,\ldots,s_m},$$
defines an algebra morphism $\hat{\Delta}_p:A_{2p}\to A_p\ot A_p$, and the family $\hat{\Delta}=\{\hat{\Delta}_p\}_{p\geq1}$
is a represented pro-morphism from $A$ to $A\ot A$. It follows from \cite[Lemma 2.11]{Sadr1} that for every $p$, $T$ induces
a linear anti-multiplicative mapping $\hat{T}_p:A_p\to A_p$, and hence, the family $\hat{T}=\{\hat{T}\}_{p\geq1}$ is a represented pro-linear
mapping from $A$ to $A$. For every $p$, the assignments $a_{p,k,l,0,\ldots,0}\mapsto\delta_{kl}$ and $a_{p,k,l,i_1,\ldots,i_m}\mapsto0$,
when $(i_1,\ldots,i_m)\neq(0,\ldots,0)$, define an algebra morphism $\hat{\epsilon}_p:A_p\to\m{K}$. Also,
$\hat{\epsilon}=\{\hat{\epsilon}_p\}_{p\geq1}$ is a represented pro-morphism from $A$ to $\m{K}$. Now, it is easily seen that $A$
is a Hopf pro-algebra with comultiplication $[\hat{\Delta}]$, antipode $[\hat{T}]$, and counit $[\hat{\epsilon}]$.
\section{ind-schemes of $G$-mappings}\label{1712031224}
Let $H\in\b{H}$ be a Hopf algebra with comultiplication $\Delta$ and counit $\epsilon$. We denote its corresponding
quantum group scheme by $G:=\f{S}H$. A \emph{$H$-comodule} is an algebra $V\in\b{A}$ with a coaction $\rho:V\to V\ot H$ satisfying
$(\rho\ot\r{id}_H)\rho=(\r{id}_V\ot\Delta)\rho$ and $(\r{id}_V\ot\epsilon)\rho=\r{id}_V$. We denote the category of $H$-comodules by $H/\b{M}$.
A morphism in $H/\b{M}$, is an algebra morphism which respects coactions. We call $G/\b{S}:=H/\b{M}^\r{op}$ the category of \emph{$G$-schemes}.
Morphisms in $G/\b{S}$ are called \emph{$G$-mappings}. Let $V,W$ be in $H/\b{M}$ with coactions $\rho:V\to V\ot H,\varrho:W\to W\ot H$.
By a \emph{(pro-)family of $H$-comodule morphisms} from $W$ to $V$, we mean a pair $(U,\phi)$ where $U$ is a (pro-)algebra and
$\phi:W\to V\ot U$ is a (pro-)algebra morphism satisfying
$$(\r{id}_V\ot F)(\rho\ot\r{id}_U)\phi=(\phi\ot\r{id}_H)\varrho,$$
where $F$ denotes flip between components of tensor product. In dual language, we can consider $(\f{S}U,\f{S}\phi)$ as
a \emph{(ind-)family of $G$-mappings} from $\f{S}V$ to $\f{S}W$, see \cite[Section 4]{Sadr1}.
\begin{lemma}\label{1712031009}
Let $D=({W},V,\{\delta_w\}_{w\in Q})$ be an object in $\b{D}$. Suppose that $V,W$ have $H$-comodule structures.
Then there exist an algebra $\tilde{\f{A}}^\dag(D)=A$ and a morphism $\tilde{\f{h}}^\dag(D)=h:{W}\to V\ot A$ satisfying the following
three properties:
\begin{enumerate}
\item[(i)] $(A,h)$ is a family of $H$-comodule morphisms from $W$ to $V$.
\item[(ii)] For every $w\in Q$, $h(w)$ is a sum of elements of the form $v\ot a$ with $v\in\delta_w,a\in A$.
\item[(iii)] The pair $(A,h)$ is universal with respect to the properties (i) and (ii).
\end{enumerate}
\end{lemma}
\begin{proof}
Similar to the proof of Lemma \ref{5-121553}.
\end{proof}
In the same manner that we proved Theorems \ref{1711241959} and \ref{5-131325} by applying Lemma \ref{5-121553}, we find the following theorems
by applying Lemma \ref{1712031009}. The proofs will be omitted for brevity.
\begin{theorem}\label{1712031023}
Let $H$ be a Hopf-algebra.
\begin{enumerate}
\item[(i)] There exist a functor $\f{A}^\dag$ and a natural transformation $\f{h}^\dag$,
$$\f{A}^\dag:H/\b{M}\times H/\b{M}^\r{op}_\r{fd}\to\b{A},\quad(W,V)\mapsto[\f{h}^\dag(W,V):W\to V\ot\f{A}^\dag(W,V)],$$
such that for every family $\phi:W\to V\ot U$ of $H$-comodule morphisms from $W$ to $V$, there is a unique morphism $\hat{\phi}:\f{A}^\dag(W,V)\to U$
satisfying $\phi=(\r{id}\ot\hat{\phi})\f{h}^\dag(W,V)$.
\item[(ii)] There exist a functor $\f{A}^\dag$ and a natural transformation $\f{h}^\dag$,
$$\f{A}^\dag:H/\b{M}_\r{fg}\times H/\b{M}^\r{op}\to\b{proA}_\r{fg},\quad(W,V)\mapsto[\f{h}^\dag(W,V):W\to V\ot\f{A}^\dag(W,V)],$$
such that for every pro-family $\phi:W\to V\ot U$ of $H$-comodule morphisms from $W$ to $V$,
there is a unique pro-morphism $\hat{\phi}:\f{A}^\dag(W,V)\to U$
satisfying $\phi=(\r{id}\ot\hat{\phi})\f{h}^\dag(W,V)$.
\end{enumerate}
\end{theorem}
In dual language, the content of Theorem \ref{1712031023} is stated as follows.
\begin{corollary}
Let $G$ be a quantum group scheme.
\begin{enumerate}
\item[(i)] We have a canonical functor $\f{M}^\dag:G/\b{S}_\r{fnt}^\r{op}\times G/\b{S}\to\b{S}$ such that for $G$-schemes
$S_1\in G/\b{S}_\r{fnt},S_2\in G/\b{S}$ there is a natural bijection between the set of $G$-mappings from $S_1$ to $S_2$
and the set of points of $\f{M}^\dag(S_1,S_2)$. Thus, $\f{M}^\dag(S_1,S_2)$ is called the \emph{scheme of $G$-mappings from $S_1$ to $S_2$}.
\item[(ii)] We have a canonical functor $\f{M}^\dag:G/\b{S}^\r{op}\times G/\b{S}_\r{fd}\to\b{indS}_\r{fd}$ such that for $G$-schemes
$S_1\in G/\b{S},S_2\in G/\b{S}_\r{fd}$ there is a natural bijection between the set of $G$-mappings from $S_1$ to $S_2$
and the set of points of $\f{M}^\dag(S_1,S_2)$. Thus, $\f{M}^\dag(S_1,S_2)$ is called the \emph{ind-scheme of $G$-mappings from $S_1$ to $S_2$}.
\end{enumerate}
\end{corollary}
\section{ind-schemes of quantum group scheme morphisms}\label{1712071728}
Let $H_1,H_2\in\b{H}$ be Hopf-algebras respectively with comultiplications $\Delta_1,\Delta_2$.
We let $G_1:=\f{S}H_1,G_2:=\f{S}H_2$ denote their associated quantum group schemes.
By a \emph{(pro-)family of Hopf-algebra morphisms} from $H_1$ to $H_2$ we mean a pair $(R,\psi)$ where $R$ is a commutative (pro-)algebra and
$\psi:H_1\to H_2\ot R$ is a (pro-)algebra morphism satisfying
$$(\Delta_2\ot\r{id}_R)\psi=(\r{id}_{H_2\ot H_2}\ot\mu_R)(\r{id}_{H_2}\ot F\ot\r{id}_R)(\psi\ot\psi)\Delta_1,$$
where $\mu_R:R\ot R\to R$ denotes the multiplication of $R$, and $F$ denotes flip between tensor product components.
(Note that since $R$ is commutative, $\mu_R$  is a (pro-)algebra morphism.)
It is natural to call $(\f{S}R,\f{S}\psi)$ a \emph{(ind-)family of group homomorphisms} from $G_2$ to $G_1$, see \cite[Section 4]{Sadr1}.
Similar to Lemma \ref{5-121553}, we have the lemma below. The following theorem is concluded by applying this lemma as in Section \ref{1712031224}.
All the proofs will be omitted for brevity.
\begin{lemma}\label{1712031229}
Let $D=({H_1},H_2,\{\delta_t\}_{t\in Q})$ be an object in $\b{D}$. Suppose that $H_1,H_2$ have Hopf-algebra structures.
Then there exist a commutative algebra $\tilde{\f{A}}^\ddag(D)=A$ and a morphism $\tilde{\f{h}}^\ddag(D)=h:{H_1}\to H_2\ot A$
satisfying the following three properties:
\begin{enumerate}
\item[(i)] $(A,h)$ is a family of Hopf-algebra morphisms from $H_1$ to $H_2$.
\item[(ii)] For every $t\in Q$, $h(t)$ is a sum of elements of the form $s\ot a$ with $s\in\delta_t,a\in A$.
\item[(iii)] The pair $(A,h)$ is universal with respect to the properties (i) and (ii).
\end{enumerate}
\end{lemma}
\begin{theorem}\label{1712031230}
\begin{enumerate}
\item[(i)] There exist a functor $\f{A}^\ddag$ and a natural transformation $\f{h}^\ddag$,
$$\f{A}^\ddag:\b{H}\times\b{H}^\r{op}_\r{fd}\to\b{A}_\r{c},\quad(H_1,H_2)\mapsto[\f{h}^\ddag(H_1,H_2):H_1\to H_2\ot\f{A}^\ddag(H_1,H_2)],$$
such that for every family $\psi:H_1\to H_2\ot R$ of Hopf-algebra morphisms,
there is a unique morphism $\hat{\psi}:\f{A}^\ddag(H_1,H_2)\to R$
satisfying $\psi=(\r{id}\ot\hat{\psi})\f{h}^\ddag(H_1,H_2)$.
\item[(ii)] There exist a functor $\f{A}^\ddag$ and a natural transformation $\f{h}^\ddag$,
$$\f{A}^\ddag:\b{H}_\r{fg}\times\b{H}^\r{op}\to\b{proA}_\r{cfg},\quad(H_1,H_2)\mapsto[\f{h}^\ddag(H_1,H_2):H_1\to H_2\ot\f{A}^\ddag(H_1,H_2)],$$
such that for every pro-family $\psi:H_1\to H_2\ot R$ of Hopf-algebra morphisms,
there is a unique pro-morphism $\hat{\psi}:\f{A}^\ddag(H_1,H_2)\to R$
satisfying $\psi=(\r{id}\ot\hat{\psi})\f{h}^\ddag(H_1,H_2)$.
\end{enumerate}
\end{theorem}
In dual language, the content of Theorem \ref{1712031230} is stated as follows.
\begin{corollary}
\begin{enumerate}
\item[(i)] We have a canonical functor $\f{M}^\ddag:\b{G}_\r{fnt}^\r{op}\times\b{G}\to\b{S}_\r{c}$ such that for quantum group schemes
$G_2\in\b{G}_\r{fnt},G_1\in\b{G}$ there is a natural bijection between the set of homomorphisms from $G_2$ to $G_1$
and the set of points of $\f{M}^\ddag(G_2,G_1)$. Thus, $\f{M}^\ddag(G_2,G_1)$ is called the
\emph{scheme of homomorphisms from $G_2$ to $G_1$}.
\item[(ii)]  We have a canonical functor $\f{M}^\ddag:\b{G}^\r{op}\times\b{G}_{fd}\to\b{indS}_\r{cfd}$ such that for quantum group schemes
$G_2\in\b{G},G_1\in\b{G}_\r{fd}$ there is a natural bijection between the set of homomorphisms from $G_2$ to $G_1$
and the set of points of $\f{M}^\ddag(G_2,G_1)$. Thus, $\f{M}^\ddag(G_2,G_1)$ is called the
\emph{ind-scheme of homomorphisms from $G_2$ to $G_1$}.
\end{enumerate}
\end{corollary}
\bibliographystyle{amsplain}

\begin{thebibliography}{10}
\bibitem{ArtinMazur1}%
M. Artin, B. Mazur,
\emph{Etale homotopy},
Lecture Notes in Mathematics 100, Springer--Verlag, 1969.
\bibitem{CherenaeckGuerra1}%
P. Cherenaeck, L. Guerra,
\emph{Spaces of morphisms between algebraic spaces},
In the mathematical heritage of CF Gauss, World Scientific (1991), 100--118.
\bibitem{Gersten1}%
S.M. Gersten,
\emph{Homotopy theory of rings},
J. Algebra, 19 no. 3 (1971), 396--415.
\bibitem{Kambayashi1}%
T. Kambayashi,
\emph{Pro-affine algebras, ind-affine groups and the Jacobian problem},
J. Algebra, 185 no. 2 (1996), 481--501.
\bibitem{KambayashiMiyanishi1}%
T. Kambayashi, M. Miyanishi,
\emph{On two recent views of the Jacobian conjecture},
Contemp. Math.,  369 (2005), 113--138.
\bibitem{KontsevichRosenberg1}%
M. Kontsevich, A.L. Rosenberg,
\emph{Noncommutative smooth spaces},
In the Gelfand mathematical seminars, 1996–-1999, pp. 85--108, Birkhäuser Boston, 2000. (arXiv:math/9812158 [math.AG])
\bibitem{Milne1}%
J.S. Milne,
\emph{Basic theory of affine group schemes},
Available online: www.jmilne.org, 2012.
\bibitem{Sadr3}%
M.M. Sadr,
\emph{Quantum functor Mor},
Math. Pannonica, 21 no. 1 (2010), 77--88. (arXiv:0807.5124 [math.OA])
\bibitem{Sadr2}%
M.M. Sadr,
\emph{A kind of compact quantum semigroups},
Int. J. Math. Math. Sci., 2012 (2012), Article ID 725270, 10 pages. (arXiv:0808.2740 [math.OA])
\bibitem{Sadr1}%
M.M. Sadr,
\emph{On the quantum groups and semigroups of maps between noncommutative spaces},
Czechoslovak Math. J., 67 no. 1 (2017): 97--121. (arXiv:1506.06518 [math.QA])
\bibitem{Shafarevich1}%
I.R. Shafarevich,
\emph{On some infinite-dimensional groups},
Rend. Mat. Appl., 25 (1966), 208–-212.
\bibitem{Shafarevich2}%
I.R. Shafarevich,
\emph{On some infinite-dimensional groups–II},
Math. USSR–Izvestija, 18 (1982), 185–-194.
\bibitem{Soltan1}%
P.M. So{\l}tan,
\emph{Quantum families of maps and quantum semigroups on finite quantum spaces},
J. Geom. Phys., 59 (2009), 354–-368. (arXiv:math/0610922 [math.OA])
\bibitem{Sweedler1}%
M.E. Sweedler,
\emph{Hopf algebras},
Mathematics Lecture Note Series, W. A. Benjamin, New York, 1969.
\bibitem{Woronowicz1}
S.L. Woronowicz,
\emph{Compact matrix pseudogroups},
Comm. in Math. Phy., 111 no. 4 (1987), 613--665.
\end{thebibliography}

\end{document}